\documentclass[11pt,reqno]{amsart}

\usepackage{amsmath,amsfonts,amsthm,amssymb,latexsym,mathrsfs,enumerate}

\usepackage{hyperref}
\usepackage[all]{xy}

\newcommand{\dist}{\mathrm{dist}}

\newcommand{\id}{\mathrm{id}}

\newcommand{\cu}{\mathrm{Cu}}
\newcommand{\cs}{\mathrm{C}^*}
\newcommand{\aff}{\mathrm{Aff}}
\newcommand{\laff}{\mathrm{LAff}}

\newtheoremstyle{smallcaps}
    {3pt}                    % Space above
    {3pt}                    % Space below
    {\itshape}                   % Body font
    {}                           % Indent amount
    {\sc}                   % Theorem head font
    {.}                          % Punctuation after theorem head
    {.5em}                       % Space after theorem head
    {}  % Theorem head spec (can be left empty, meaning ÔnormalÕ)
    
\newtheoremstyle{smallcapsdef}
    {3pt}                    % Space above
    {3pt}                    % Space below
    {}                   % Body font
    {}                           % Indent amount
    {\sc}                   % Theorem head font
    {.}                          % Punctuation after theorem head
    {.5em}                       % Space after theorem head
    {}  % Theorem head spec (can be left empty, meaning ÔnormalÕ)

\theoremstyle{smallcaps}

\newtheorem {theorem}{Theorem}[section]
\newtheorem {lemma}[theorem]{Lemma}

\newtheorem {corollary}[theorem]{Corollary}

\theoremstyle {smallcapsdef}

\newtheorem {remark}[theorem]{Remark}

\numberwithin{equation}{section}

\begin{document}

\author{Bhishan Jacelon, Karen R. Strung and Andrew S. Toms}
\date{\today}
\title[Unitary orbits in simple $\mathcal{Z}$-stable $\cs$-algebras]{Unitary orbits of self-adjoint operators in simple $\mathcal{Z}$-stable $\cs$-algebras}

\address{Department of Mathematics\\Purdue University\\150 North University Street\\West Lafayette, IN 47907\\USA} 
\email{bjacelon@purdue.edu}
\address{Instytut Matematyczny\\Polskiej Akademii Nauk\\ul. \'{S}niadeckich 8\\00-656 Warszawa\\Poland} 
\email{kstrung@impan.pl}
\address{Department of Mathematics\\Purdue University\\150 North University Street\\West Lafayette, IN 47907\\USA} 
\email{atoms@purdue.edu}
\thanks{Research supported by NSF DMS-1302763}


\maketitle

\begin{abstract}
We prove that in a simple, unital, exact, $\mathcal{Z}$-stable $\cs$-algebra of stable rank one, the distance between the unitary orbits of self-adjoint elements with connected spectrum is completely determined by spectral data. This fails without the assumption of $\mathcal{Z}$-stability.
\end{abstract}

\section{Introduction}

It is well known that the focus of the Elliott classification programme has shifted from the class of all separable nuclear $\cs$-algebras to, in light of the counterexamples \cite{Toms:2008hl} and \cite{Rordam:2003rz}, those that are sufficiently well behaved. It has subsequently become a natural pursuit to unify the various notions of `regularity' for $\cs$-algebras, most notably in the form of the Toms--Winter conjecture (see for example \cite{Winter:2012pi}). Furthermore, any meaningful notion of regularity entails not only amenability to classification but also, inevitably, tameness of internal structure. In particular, there is a persistent expectation of mimicry in regular $\cs$-algebras of the good behaviour found in the world of von Neumann factors (see especially \cite{Sato:2014aa} for an exposition of this point of view). This article further develops that theme.

One of the most fundamental questions one can ask in an operator algebra is: when are two operators unitarily equivalent? For normal operators, this problem is inextricably linked with spectral theory, and we refer to \cite{Sherman:2007aa} for a discussion of its history. The present article addresses the related problem, which has an equally storied past, of computing the distance $d_U$ between unitary orbits  in terms of spectral information. For normal matrices $a,b\in M_n$, one asks whether $d_U(a,b)$ is equal to the \emph{optimal matching distance}
\begin{equation} \label{matching}
\delta(a,b) = \min_{\sigma\in S_n} \max_{1\le i \le n} |\alpha_i - \beta_{\sigma(i)}|
\end{equation}
where $\alpha_1,\ldots,\alpha_n$ and $\beta_1,\ldots,\beta_n$ are the eigenvalues of $a$ and $b$ respectively. If $a$ and $b$ are self-adjoint then a classical result of H.~Weyl \cite{Weyl:1912aa} says that these distances do indeed agree (and moreover $\delta(a,b)$ can be measured by listing the eigenvalues in ascending order). This continues to hold for example for unitary matrices \cite{Bhatia:1984aa} but not necessarily for normal matrices \cite{Holbrook:1992aa} (although $d_U$ and $\delta$ are known in general to be Lipschitz equivalent \cite{Bhatia:1983aa}).

Redefining $\delta$ as a \emph{crude multiplicity function} described in terms of ranks of spectral projections, this was extended in \cite{Azoff:1984aa} to self-adjoint operators on infinite dimensional Hilbert space, and further in \cite{Hiai:1989aa} to self-adjoint elements in a $\sigma$-finite semifinite factor. In the latter case, one defines $\delta(a,b)$ for normal elements $a$ and $b$ to be the infimum over $r>0$ such that 
\begin{equation} \label{crude}
\tau(\chi_{U_r}(a)) \ge \tau(\chi_{U}(b)) \quad\textrm{and}\quad \tau(\chi_{U_r}(b)) \ge \tau(\chi_{U}(a))
\end{equation}
for every open subset $U\subset\mathbb{C}$. Here $\tau$ is a fixed faithful normal semifinite trace, $U_r = \{t\mid \dist(t,U)<r\}$ and   $\chi_V$ is the indicator function on the open subset $V\subset \mathbb{C}$.

The $\cs$-regularity property most relevant to this article is that of `$\mathcal{Z}$-stability', that is, tensorial absorption of the Jiang--Su algebra $\mathcal{Z}$ \cite{Jiang:1999hb}.  A simple, unital, $\mathcal{Z}$-stable $\cs$-algebra is either stably finite or purely infinite \cite{Gong:2000kq}. In the latter case, it is shown in \cite{Skoufranis:2013aa} that (modulo $K$-theory) the distance $d_U(a,b)$ between the normal operators $a$ and $b$ is simply the Hausdorff distance between their spectra. This mirrors the corresponding result obtained for type $\rm{III}$ factors in \cite{Hiai:1989aa}, and this article may therefore be regarded as a step towards completion of the analogy. The appropriate $\cs$-analogue of (\ref{crude}) is the L\'evy--Prokhorov distance $d_P$ discussed in Section~\ref{lp} below, and we do indeed obtain that $d_U(a,b)=d_P(a,b)$ for self-adjoint elements with connected spectra in the stably finite setting. (By translating by a multiple of the unit, there is no loss of generality in restricting to positive elements.)

Our strategy along the way is to exploit another notion of spectral distance, the pseudometric $d_W$ defined in \cite{Ciuperca:2008rz} in terms of Cuntz equivalence (see Section~\ref{prelim}), and to bootstrap the equality $d_U=d_W$ from matrices to a class of algebras that exhausts the tracial invariant (Section~\ref{razak}) and ultimately via classification to the class of simple $\mathcal{Z}$-stable $\cs$-algebras of stable rank one (Section~\ref{zstable}), a class for which we also obtain $d_W=d_P=d_U$ (Section~\ref{lp}).

Finally, it should be noted that the assumption of connected spectra is necessary for a purely measure theoretic calculation of $d_U$. There exist, for example, nonequivalent projections in a simple, unital, monotracial AF algebra that have the same trace (see \cite[7.6.2]{Blackadar:1998qf}). Moreover, this computation really does rely on regularity of the ambient $\cs$-algebra. See \cite[Section 4.3]{Robert:2010rz} for an example of a simple (necessarily non-$\mathcal{Z}$-stable) AH algebra where $d_U$ and $d_P$ differ: there exist full spectrum positive contractions $a,b$ that are not approximately unitarily equivalent, so $d_U(a, b)\neq 0$, but for which $d_W(a,b)=0$ (hence also $d_P(a,b)=0$).

\section{Preliminaries} \label{prelim}

Let $A$ be a $\cs$-algebra. We will denote the cone of positive elements in $A$ by $A^+$, the minimal unitisation of $A$ by $\tilde{A}$, and the group of unitaries in $\tilde{A}$ by $\mathcal{U}(\tilde{A})$.  

Given a positive element $x \in A$, let $e_s(x)=(x-s)_+$, that is, the element of $\cs(x)$ corresponding under functional calculus to $f(t)=\max\{0,t-\varepsilon\}$. Two positive elements $x, y \in A$ are Cuntz subequivalent, written $x \precsim y$ if $\|x-v_nyv_n^*\| \to 0$ for some $v_n\in A\otimes\mathcal{K}$, and denote by $\sim$ the relation that symmetrises $\precsim$. (We refer the reader to \cite{Ara:2009cs} for more information on the Cuntz semigroup $\cu(A)=(A\otimes\mathcal{K})^+/\sim$.) 

We consider the following pseudometrics on the set of positive elements of $A$:
\begin{equation}
d_U(a,b)=\inf\{\|uau^*-b\| \mid u\in\mathcal{U}(\tilde A)\}
\end{equation}
(so $d_U(a,b)=0$ if and only if $a$ and $b$ are approximately unitarily equivalent), and
\begin{equation} \label{cuntz}
d_W(a,b)=\inf\{r>0 \mid e_{t+r}(a)\precsim e_t(b) \; \textrm{and}\; e_{t+r}(b)\precsim e_t(a) \: \forall t > 0\}.
\end{equation}
It is perhaps an instructive exercise to convince oneself that (\ref{matching}) and (\ref{cuntz}) agree for positive matrices.

The following, which implies in particular that $d_U$ and $d_W$ are continuous, makes its initial appearance as \cite[Corollary 9.1]{Ciuperca:2008rz}. See \cite[Lemma 1]{Robert:2010rz} for a straightforward functional calculus proof.

\begin{lemma}\label{cts}
For $a,b\in A^+$, we have $d_W(a,b)\le d_U(a,b) \le \|a-b\|$.
\end{lemma}

The main result of \cite{Ciuperca:2008rz} is that, if $A$ has stable rank one, then
\begin{equation}
d_W(a,b)\le d_U(a,b) \le 8d_W(a,b).
\end{equation}
In particular, two $^*$-homomorphisms from $C_0(0,1]$ to $A$ are approximately unitarily equivalent if and only if they agree at the level of the Cuntz semigroup. This inequality (with the factor of $8$ improved to $4$) has been demonstrated in \cite{Robert:2010rz} for a class of $\cs$-algebras strictly larger than that of the stable rank one algebras, provided that $d_U$ is computed in the stabilisation (note that this does not make a difference in the stable rank one case).

We are interested in showing that $d_U(a,b) = d_W(a,b)$ for every $a,b\in A^+$ when $A$ is sufficiently well behaved. This is the case, for example for inductive limits $\varinjlim C(X_i)$, where the $X_i$ are compact Hausdorff spaces of topological dimension at most $2$ with $\check{H}^2(X_i)=0$ (see \cite[Proposition 5]{Robert:2010rz}).

\begin{remark} \label{dense}
By Lemma~\ref{cts}, to show that $d_U=d_W$ it suffices to find a dense subset $A'$ of $A^+$ such that $d_U(a,b)\le d_W(a,b)$ for every $a,b\in A'$.
\end{remark}

\begin{lemma}\label{perm}
The property ``$d_U=d_W$'' passes to finite direct sums and sequential inductive limits.
\end{lemma}

\begin{proof}
See \cite[Lemma 5]{Robert:2010rz} for inductive limits. Finite direct sums are straightforward.
\end{proof}

We will first show, in Section~\ref{razak}, that certain stably projectionless algebras satisfy $d_U=d_W$. We then use classification in Section~\ref{zstable} to obtain this for simple, $\mathcal{Z}$-stable algebras for positive contractions with full spectrum.

\section{Razak blocks} \label{razak}

A \emph{Razak block} is a $\cs$-algebra of the form
\begin{equation} \label{block}
\{f \in C([0,1], M_k\otimes M_n) \mid f(0)=c\otimes 1_{n-1}, f(1)=c\otimes 1_n, c\in M_k\},
\end{equation}
and is regarded as a subalgebra of $C([0,1],M_{kn})$. Such algebras are stably projectionless and have trivial $K$-theory. Simple inductive limits of finite direct sums of Razak blocks are classified by tracial data (see \cite{Razak:2002kq} and also \cite{Robert:2010qy}), and as a consequence are UHF-stable, hence $\mathcal{Z}$-stable. (Alternatively, since such limits have nuclear dimension one, $\mathcal{Z}$-stability follows from \cite{Tikuisis:2012kx}.) 

\begin{theorem} \label{nccw}
Let $A$ be a Razak block as in (\ref{block}), and let $a,b\in A^+$. Then $d_U(a,b) = d_W(a,b)$.
\end{theorem}

\begin{proof}
The proof is an adaptation of \cite[Theorem 4.1.6]{Cheong:2013} (see also \cite{Cheong:2013aa}). By Remark~\ref{dense}, we may assume that, for some $\gamma>0$, both $a$ and $b$ are constant on $[0,\gamma]$ and $[1-\gamma,1]$. Moreover, it suffices to show that, if $d_W(a,b) <r$ and $\varepsilon>0$, then there is a unitary $w\in \mathcal{U}(\tilde A)$ such that
\[
\|w_ta_tw_t^*-b_t\|<r+\varepsilon \quad \textrm{for every } t\in[0,1].
\]

So let us take such $r$ and $\varepsilon$. Let $\delta\in(0,\varepsilon/2)$. For every $t \in [0,1]$, we have $d_W(a_t,b_t) \le d_W(a,b) <r$, which means that the eigenvalues of $a_t$ and $b_t$, listed in ascending order, are paired within distance $r$ of each other. We will use this to construct $w$ on $[0,\gamma]$, $[\gamma,1-\gamma]$ and $[1-\gamma,1]$.\\

\noindent\underline{On $[\gamma, 1-\gamma]$:} Let $D^a_t$ (respectively $D^b_t$) be the diagonal matrix whose diagonal entries are the eigenvalues of $a_t$ (respectively $b_t$) in ascending order. By \cite[Lemma 1.1]{Thomsen:1992qf} and \cite[Corollary 1.3]{Thomsen:1992qf}, $D^a$ and $D^b$ are continuous, and there exist unitaries $u,v\in C([\gamma,1-\gamma], M_{kn})$ such that for every $t\in[\gamma,1-\gamma]$,
\begin{equation}
\|u_ta_tu_t^* - D^a_t\|<\delta \quad\textrm{and}\quad \|v_tb_tv_t^* - D^b_t\|<\delta.
\end{equation}
Define
\begin{equation}
w(t)=v_t^*u_t \quad\textrm{for } t\in[\gamma,1-\gamma]. 
\end{equation}
Then for such $t$ we have
\[
\|w_ta_tw_t^*-b_t\| \le \|D^a_t-D^b_t\| +2\delta <r+\varepsilon.
\]
\underline{On $[0,\gamma]$:} Choose unitaries $U,V\in M_k$ such that, with $u_0=(U\otimes1_{n-1})\oplus(1_k\otimes e_{nn})$ and $v_0=(V\otimes1_{n-1})\oplus(1_k\otimes e_{nn})$, each $M_k$ block of $u_0a_0u_0^*$ and $v_0b_0v_0^*$ is diagonal with the eigenvalues in ascending order. In particular,
\begin{equation}
\|u_0a_0u_0^* - v_0b_0v_0^*\| < r.
\end{equation}
Choose a permutation matrix $x$ such that $xu_0a_0u_0^*x^*$ and $xv_0b_0v_0^*x^*$ are diagonal matrices whose entries appear in ascending order (the same $x$ works for both matrices). That is, since $a_0=a_\gamma$,
\begin{equation}
(xu_0)a_0(xu_0)^* = D^a_\gamma \quad\textrm{and}\quad (xv_0)b_0(xv_0)^* = D^b_\gamma.
\end{equation}
Then
\begin{equation}
\|[(xu_0)^*u_\gamma, a_0]\| = \|u_\gamma a_0 u_\gamma^*- (xu_0)a_0(xu_0)^*\| = \|u_\gamma a_\gamma u_\gamma^* - D^a_\gamma\| < \delta,
\end{equation}
and similarly $\|[(xv_0)^*v_\gamma, b_0]\| < \delta$. By \cite[Lemma 2.6.11]{Lin:2001it}, which is an exercise in functional calculus, provided we chose $\delta$ sufficiently small to begin with, there are therefore paths $f,g\in C([0,\gamma],\mathcal{U}_{kn})$ from $1$ to $(xu_0)^*u_\gamma$ and $(xv_0)^*v_\gamma$ respectively, such that
\begin{equation}
\|[f_t,a_0]\|<\epsilon/2 \quad\textrm{and}\quad \|[g_t,b_0]\|<\epsilon/2 \quad\textrm{for every } t\in[0,\gamma].
\end{equation}
Define
\begin{equation}
w_t=g_t^*(xv_0)^*(xu_0)f_t \quad\textrm{for } t\in[0,\gamma].
\end{equation}
Note that $g_\gamma^*(xv_0)^*(xu_0)f_\gamma = v_\gamma^*u_\gamma$, so now $w$ is well defined and continuous on $[0,1-\gamma]$. Moreover,
\begin{equation}\label{t0}
w_0=v_0^*u_0 =(V^*U\otimes 1_{n-1})\oplus(1_k\otimes e_{nn}).
\end{equation}
Finally, for $t\in[0,\gamma]$ we have
\begin{align*}
\|w_ta_tw_t^* - b_t\| &= \|(xu_0)(f_ta_tf_t^*)(xu_0)^* - (xv_0)(g_tb_tg_t^*)(xv_0)^*\|\\
&\le \|(xu_0)a_t(xu_0)^* - (xv_0)b_t(xv_0)^*\| + \varepsilon\\
&=\|u_0a_0u_0^* - v_0b_0v_0^*\| + \varepsilon\\
&< r+\varepsilon.
\end{align*}
\underline{On $[1-\gamma,1]$:} By exactly the same argument, we extend $w$ continuously to $[1-\gamma,1]$ with $\|w_ta_tw_t^*-b_t\|<r+\varepsilon$ for every $t$ and with
\begin{equation}\label{t1}
w_1=V^*U\otimes1_n.
\end{equation}
(This is because $U\otimes1_n$ and $V\otimes1_n$ play the diagonalising roles at $t=1$ that $u_0=(U\otimes1_{n-1})\oplus(1_k\otimes e_{nn})$ and $v_0=(V\otimes1_{n-1})\oplus(1_k\otimes e_{nn})$ did at $t=0$. A different permutation matrix $x$ may be needed but this does not matter.) Comparing (\ref{t0}) and (\ref{t1}) we see that the unitary $w$ we have constructed is indeed in $\tilde A$ (because $w-1\in A$), so we are done.
\end{proof}

\begin{corollary}\label{cor}
If $A$ is a sequential inductive limit of finite direct sums of Razak blocks, then $d_U(a,b)=d_W(a,b)$ for every $a,b\in A^+$.
\end{corollary}

\begin{remark}\label{cheong}
In the unital case Theorem~\ref{nccw} holds for more general type I $\cs$-algebras. In particular, Cheong has shown in \cite{Cheong:2013aa}  that it is true for any one-dimensional NCCW complex. 
\end{remark}

\section{Simple, $\mathcal{Z}$-stable $\cs$-algebras} \label{zstable}

The proof of our main theorem, that $d_U$ and $d_W$ agree for full spectrum positive contractions in simple $\mathcal{Z}$-stable $\cs$-algebras, relies on the structure of the Cuntz semigroup of these algebras and on the classification established in \cite{Robert:2010qy}. A preliminary discussion seems warranted.

Building on \cite{Brown:2008mz} and \cite{Brown:2007rz}, it is shown in \cite{Elliott:2009kq} that if $A$ is simple, stably finite, $\mathcal{Z}$-stable and exact, then
\begin{equation}\label{cu}
\cu(A) \cong V(A)\backslash\{0\} \sqcup \laff(T(A)),
\end{equation}
where $V(A)$ is the Murray--von Neumann semigroup of projections over $A$, $T(A)$ is the cone of densely finite lower semicontinuous traces on $A$, and $\laff(T(A))$ denotes the union of the zero functional and those lower semicontinuous linear functionals that are suprema of increasing sequences of elements in
\begin{displaymath}
\aff(T(A))=\{f:T(A)\to\mathbb{R}^+ \mid \textrm{$f$ linear, continuous}, f|_{T(A)\backslash\{0\}}>0\}.
\end{displaymath}
(The isomorphism is obtained by sending $[a]$ to $[p]\in V(A)$ if $a$ is Cuntz equivalent to a nonzero projection $p$, and otherwise to the functional $\widehat{[a]}$ on $T(A)$ defined by $\widehat{[a]}(\tau) = d_\tau(a) = \lim_{n\to\infty}\tau(a^{1/n})$. Note that if $e\in A$ is strictly positive, then $\widehat{[e]}(\tau)=\|\tau\|$ for every $\tau\in T(A)$.)

The main theorem of \cite{Robert:2010qy} is that the finer invariant $\cu^\sim$ classifies $^*$-homomorphisms from inductive limits of one-dimensional NCCW complexes with trivial $K_1$ (a class of $\cs$-algebras that includes in particular those considered in Corollary~\ref{cor}) to stable rank one $\cs$-algebras. We refer the reader to \cite{Robert:2010qy} for the definition but note in particular that (in the stable rank one case), $\cu^\sim(A)$ contains $\cu(A)$ as its positive cone. If moreover $A$ is simple and $\mathcal{Z}$-stable then it follows from (\ref{cu}) that
\begin{equation} \label{cusim}
\cu^\sim(A) \cong K_0(A)\backslash\{0\} \sqcup \laff^\sim(T(A))
\end{equation}
where
\begin{equation}
\laff^\sim(T(A))= \laff(T(A)) - \aff(T(A))
\end{equation}
(see \cite[Section 6]{Robert:2010qy} for details).

\begin{theorem} \label{main}
Let $A$ be simple, separable, $\mathcal{Z}$-stable, exact and of stable rank one.  Then $d_W(a,b) = d_U(a,b)$ for every  $a,b\in A^+$ with $\sigma(a)=\sigma(b)=[0,1]$.
\end{theorem}

\begin{proof}
The argument is essentially the same as is laid out in \cite[Section 3.4]{Strung:2013aa} (except that we do not have to worry about $K$-theory). By (\ref{cu}) and (\ref{cusim}), we have
\begin{equation}
\cu(A) \cong V(A)\backslash\{0\} \sqcup \laff(T(A))\:,\:\cu^\sim(A) \cong K_0(A)\backslash\{0\} \sqcup \laff^\sim(T(A)).
\end{equation}
Let $s_A\in A$ be strictly positive. By \cite[Proposition 5.3]{Jacelon:2010fj}, there exists a simple inductive limit $B$ of finite direct sums of Razak blocks together with a strictly positive element $s_B\in B$ and an isomorphism $\gamma:T(A)\to T(B)$ under which $\widehat{[s_A]}$ corresponds to $\widehat{[s_B]}$. (Essentially, $B$ is the tensor product of the $\cs$-algebra $\mathcal{W}$, a simple monotracial inductive limit of Razak blocks, with an appropriate AF-algebra.) Since $B$ is $\mathcal{Z}$-stable and $K_0(B)=0$ (see Section~\ref{razak}), we have
\begin{equation}
\cu(B) = \laff(T(B))\:,\:\cu^\sim(B) = \laff^\sim(T(B)).
\end{equation}
The obvious morphism
\begin{equation}
\theta: \cu^\sim(B) \to \laff^\sim(T(A)) \subset \cu^\sim(A),
\end{equation}
that maps $\cu(B)$ onto $\laff(T(A))$ via
\begin{equation} \label{theta}
\theta([b])(\tau) = \hat{[b]}(\gamma(\tau)) = d_{\gamma(\tau)}(b) \quad (\tau\in T(A)),
\end{equation}
satisfies $\theta([s_B])=\widehat{[s_A]}\le [s_A]$ and therefore, by \cite[Theorem 1.0.1]{Robert:2010qy}, lifts to an embedding $\iota: B \to A$ with $\cu^\sim(\iota)=\theta$.

Next, since no nonzero element of $\cs(a)$ is Cuntz equivalent to a projection, the inclusion $\cs(a)\hookrightarrow A$ induces a morphism
\[
\cu(\cs(a)) \to \laff(T(A)) \subset \cu(A),
\]
hence a morphism
\begin{equation}
\varphi:\cu(C_0(0,1]) \to \cu(B),
\end{equation}
namely
\begin{equation}
\varphi([f])(\tau) = d_{\gamma^{-1}(\tau)}(f(a)) \quad (f\in C_0(0,1]^+, \tau\in T(B)),
\end{equation}
that moreover satisfies $\varphi([\id])\le[s_B]$. By \cite[Theorem 1]{Ciuperca:2011wd} we can lift this to a $^*$-homomorphism $C_0(0,1] \to B$. In other words, we can find a positive contraction $a'\in B$ such that
\begin{equation} \label{ces}
d_{\gamma(\tau)}(f(a')) = d_\tau(f(a)) \quad (f\in C_0(0,1]^+, \tau\in T(A)).
\end{equation}
Then for every $f\in C_0(0,1]^+$ we have (in $\cu(A)$)
\begin{equation}
[\iota(f(a'))] = \theta([f(a')]) = [f(a)]
\end{equation}
by (\ref{theta}) and (\ref{ces}).

Therefore, regarding $B$ as a subalgebra of $A$, we have $d_W(a,a')=0$. Since $A$ has stable rank one, we have $d_U(a,a')=0$ as well. Similarly we find a positive contraction $b'\in B$ with $d_U(b,b')=0$. Then, specifying the algebra in which the relevant distance should be measured,
\begin{align*}
d^A_W(a,b) &= d^A_W(a',b')\\
&= d^B_W(a',b') \quad(\ref{theta})\\
&= d^B_U(a',b') \quad\textrm{(Corollary~\ref{cor})}\\
&\ge d^A_U(a',b') \quad(\mathcal{U}(\tilde B)\subset \mathcal{U}(\tilde A))\\
&= d^A_U(a,b),
\end{align*}
hence (see Remark~\ref{dense}), $d_W(a,b)=d_U(a,b)$.
\end{proof}

\begin{remark} We note here that the hypotheses of Theorem~\ref{main} can be relaxed.
\begin{enumerate}[(i)]
\item A simple, unital, $\mathcal{Z}$-stable $\cs$-algebra either has stable rank one or is purely infinite (\cite{Gong:2000kq} and \cite[Theorem 6.7]{Rordam:2004kq}). If $A$ is unital and purely infinite, then $d_U(a,b)=0$ by \cite[Lemma 2.11]{Skoufranis:2013aa} (or indeed from the very general \cite[Theorem 1.7]{Dadarlat:1995uq}) and $d_W(a,b)=0$ more or less by definition. So, at least in the unital case, the assumption of stable rank one can be dropped. (A suitable generalisation of \cite[Theorem 6.7]{Rordam:2004kq}, which says that stable rank one for simple, unital, stably finite, $\mathcal{Z}$-stable algebras is automatic, is not known for  stably projectionless algebras.)

\item By considering $2$-quasitraces rather than traces, Theorem~\ref{main} is probably true without the assumption of exactness.
\item Recall that a simple, separable $\cs$-algebra is \emph{pure} if its Cuntz semigroup has strict comparison and almost divisibility (see \cite[Definition 3.6]{Winter:2012pi}). In particular, simple, separable, $\mathcal{Z}$-stable algebras are pure (see \cite{Rordam:2004kq} and also \cite[Proposition 3.7]{Winter:2012pi}, which is stated for unital algebras but does not use this assumption). The computations (\ref{cu}) and (\ref{cusim}) hold not just for the class of (simple, separable) $\mathcal{Z}$-stable algebras, but for the strictly larger class of pure $\cs$-algebras that have stable rank one (see \cite[Section 4]{Ng:2014aa} for details, and the references therein). So Theorem~\ref{main} holds for pure, stable rank one $\cs$-algebras as well.
\end{enumerate}
\end{remark}

\section{The L\'evy--Prokhorov distance} \label{lp}

In this section we provide the $\cs$-analogue of \cite[Theorem 1.3]{Azoff:1984aa}, namely, we compute the distance between the unitary orbits of self-adjoint elements  in terms of spectral data. We first need a $\cs$-version of (\ref{crude}).

Let $A$ be a unital, stably finite, exact $\cs$-algebra. For positive contractions $a,b\in A$, define $d_P(a,b)$ to be the infimum over $r>0$ such that
\begin{equation}
\mu_{\tau,a}(U_r) \ge \mu_{\tau,b}(U) \quad \textrm{and}\quad \mu_{\tau,b}(U_r) \ge \mu_{\tau,a}(U)
\end{equation}
for every open subset $U\subset (0,1]$ and every trace $\tau\in T(A)$ (where $\mu_{\tau,a}$ and $\mu_{\tau,b}$ denote the Borel measures induced by $\tau$ on $\cs(a)$ and $\cs(b)$ respectively, and $U_r = \{t\mid \dist(t,U)<r\}$).

Note that if $U\subset(0,1]$ is open and the support of $f\in C_0(0,1]^+$ is $U$, then $\mu_{\tau,a}(U)=d_\tau(f(a))$. In particular, $d_P(a,b)=0$ if and only if $d_\tau(f(a)) = d_\tau(f(b))$ for every $\tau\in T(A)$ and $f\in C_0(0,1]^+$.

The following is contained in Lemmas 1 and 2 of \cite{Cheong:2013aa}.

\begin{lemma} \label{levy}
Let $A$ be unital, simple, exact and stably finite, and suppose moreover that $A$ has strict comparison (for example, $A$ might be $\mathcal{Z}$-stable). Then for every $a,b\in A^+$ with $\sigma(a)=\sigma(b)=[0,1]$, we have
\begin{equation}
d_W(a,b)\le d_P(a,b) \le d_U(a,b).
\end{equation}
\end{lemma}

\begin{proof}
The inequality $d_P\le d_U$ is proved exactly as in \cite[Lemma 1]{Robert:2010rz}, and no assumption on spectra is needed. The condition $\sigma(a)=[0,1]$ ensures that no nonzero element of $\cs(a)$ is Cuntz equivalent to a projection (and similarly for $b$), and then the inequality $d_W\le d_P$ is readily obtained by appealing to a version of strict comparison not available to projections, namely: $x\precsim y$ if and only if $d_\tau(x) \le d_\tau(y)$ for every $\tau$ (see \cite[Proposition 5.9]{Ara:2009cs}).
\end{proof}

Combining Theorem~\ref{main} and Lemma~\ref{levy}, we have the following.
\begin{theorem}
Let $A$ be unital, simple, separable, stably finite, $\mathcal{Z}$-stable and exact, and let $a,b\in A^+$ with $\sigma(a)=\sigma(b)=[0,1]$. Then
\[
d_U(a,b) = d_P(a,b) = d_W(a,b).
\]
\end{theorem}

In particular, $a$ and $b$ are approximately unitarily equivalent if and only if $d_\tau(f(a)) = d_\tau(f(b))$ for every $\tau\in T(A)$ and $f\in C_0(0,1]^+$. This has been obtained without the assumption of simplicity in \cite[Theorem 1.3]{Robert:2013aa}.

\section{Outlook} \label{outlook}

Here are two avenues open for continued study.

\begin{enumerate}[(i)]
\item What can be said of \emph{normal} operators in a simple, $\mathcal{Z}$-stable $\cs$-algebra of stable rank one? Or in other words, full spectrum positive contractions represent injective $^*$-homomorphisms from $C_0(0,1]$. How far can one extend our results from $[0,1]$ to compact metric spaces $X$ (compared for example to \cite{Dadarlat:1995uq} and \cite{Skoufranis:2013aa})? In light of what happens in the world of semifinite factors, one would expect, at best, inequalities.
\item Extending another way, a positive contraction in $A$ represents an `order zero' map from $\mathbb{C}$ to $A$. What can be said about order zero maps from, for example, finite dimensional algebras? This question will be addressed in subsequent work.
\end{enumerate}

\section*{Acknowledgements} We are indebted to Michael Cheong, Caleb Eckhardt, Leonel Robert and Stuart White for many helpful discussions and contributions. The second author would like to thank the Fields Institute for support during the 2014 Thematic Program on Abstract Harmonic Analysis, Banach and Operator Algebras, where much of the work was carried out. The third author would like to acknowledge the National Science Foundation research grant DMS-1302763.

\end{document}